\documentclass[10pt]{article}
\makeatletter
\newfont{\footsc}{cmcsc10 at 8truept}
\newfont{\footbf}{cmbx10 at 8truept}
\newfont{\footrm}{cmr10 at 10truept}
\makeatother
\newcommand{\half}{{\textstyle\frac12}}

\newcommand{\ignore}[1]{}
\newcommand {\Z}  {{{\sf Z}\hspace{-1.3mm}{\sf Z}}}

\newtheorem{theorem}{Theorem}
\newtheorem{lemma}{Lemma}

\newtheorem{proposition}{Proposition}

\newenvironment{proof}{\begin{trivlist}
                       \item[]{\bf Proof}
                       \hspace{0cm} }{\hfill {\large $\bullet$}
                       \end{trivlist}}

\newenvironment{acknowledgment}{\begin{trivlist}
                         \item[]{\bf Acknowledgment} \\
                        }{\end{trivlist}}
\newenvironment{remark}{\begin{trivlist}
                         \item[]{\bf Remark} \\
                        }{\end{trivlist}}

\def\quad{ ~ }
\def\qquad{ ~~ }

\input{epsf}

\def\sp{\def\baselinestretch{1.2}\large\normalsize}
\def\dsp{\def\baselinestretch{1.8}\large\normalsize}

\title{A prime sensitive Hankel determinant of \\
Jacobi symbol enumerators}  

\author{\"Omer E\u{g}ecio\u{g}lu \\
Department of Computer Science, \\
 University of California,\\
 Santa Barbara CA 93106 \\   
 ({\tt omer@cs.ucsb.edu})
 }

\begin{document}

\date{}

\maketitle

\sp

\begin{abstract}

We show that the determinant of a Hankel matrix of odd dimension $n$ whose entries are 
the enumerators of the Jacobi symbols which depend on the row and the column indices 
vanishes iff $n$ is composite.
If the dimension is a prime $p$, then the determinant
evaluates to a polynomial of degree
$p-1$ which is the product of a power of $p$ and
the generating polynomial
of the partial sums of Legendre symbols. The sign of the 
determinant is determined by the quadratic character of
$-1$ modulo $p$.

The proof of the evaluation 
makes use of elementary properties of Legendre symbols,  quadratic Gauss sums and 
orthogonality of trigonometric functions.\\


\end{abstract}

\section{Introduction}

For an odd integer $n$, and $ k =1, 2, \ldots , n $ define the polynomials
$$
a_k (x) = \sum_{m=0}^k J({k-m},{n}) x^m
$$
in which $J({a},{m}) $ is the Jacobi symbol defined for odd integers $m$ by
$$
J({a},{m}) = 
\Big( \frac{a}{p_1} \Big)^{e_1} 
\Big( \frac{a}{p_2} \Big)^{e_2} 
\cdots
\Big( \frac{a}{p_k} \Big)^{e_k} 
$$
where the prime factorization of $m$ is
$m= p_1^{e_1} p_2^{e_2} \cdots p_k^{e_k}$
and
for a prime $p$, 
$\Big( \frac{a}{p} \Big)$ is the Legendre symbol defined by
$$
 \Big( \frac{a}{p} \Big) = 
\left\{
\begin{array}{rl}
0 & \mbox{if} ~ p \, | ~ a ,\\
1 & \mbox{if} ~a~ \mbox{is a quadratic residue mod } p , \\
-1 & \mbox{if} ~a~ \mbox{is a quadratic nonresidue mod } p ~. 
\end{array}
\right.
$$

For example when $ n =3$,
the first five polynomials are
\begin{eqnarray*}
 a_1(x) &=& 1 \\
 a_2(x) &=& x-1 \\
 a_3(x) &=& x^2-x \\
 a_4(x) &=& x^3-x^2+1 \\
 a_5(x) &=& x^4-x^3+x-1
\end{eqnarray*}
It is easy to see that $a_k(x)$ is a monic polynomial of 
degree $k-1$ and $a_k(0) = J(k,n)$. 
Consider the  $n \times n$ Hankel determinant
\begin{equation}
\label{Hn}
H_n(x) = \det [ a_{i+j-1} (x)]_{1 \leq i,j \leq n} ~.
\end{equation}

As an example, 
$$
H_3 (x) = \det \left[
\begin{array}{lll}
 1 & x-1 & x^2-x \\
 x-1 & x^2-x & x^3-x^2+1 \\
 x^2-x & x^3-x^2+1 & x^4-x^3+x-1
\end{array}
\right] 
= - x^2 ~.
$$
A few other determinant evaluations for small $n$ are as follows:
\begin{eqnarray*}
H_5(x) &=& 5 x^2 (x-1)  (x+1) \\
H_7(x) &=& -49 x^2 (x^4+2 x^3+x^2+2 x+1) \\
H_9(x) &=& 0 \\
H_{11}(x) & = & -14641 x^2 (x^8+x^6+2 x^5+3 x^4+2 x^3+x^2+1) \\
H_{13} (x) &=& 371293 x^2 (x-1) (x+1) (x^8+2 x^6+2 x^5+3 x^4+2 x^3+2 x^2+1) \\
H_{15} (x) &=&  0\\
H_{17} (x) &=& 410338673 x^2 (x-1) (x+1) \\
& & \hspace*{5mm} (x^{12}+2 x^{11}+2 x^{10}+4 x^9+3 x^8+4 x^7+2 x^6+4 x^5+3 x^4+4 x^3+2 x^2+2 x+1)  \\
H_{19}(x) &=& -16983563041 x^2 \\
& & \hspace*{5mm} (x^{16}-x^{14}+x^{12}+2 x^{11}+3 x^{10}+2 x^9+3 x^8+2 x^7+3 x^6+2 x^5+x^4-x^2+1) 
\end{eqnarray*}

Recently, Chapman \cite{Chapman03} evaluated Hankel determinants of 
certain $ \frac{p-1}{2} \times \frac{p-1}{2}$ dimensional 
0-1 matrices built up from the Legendre symbol defined modulo a prime $p$. These evaluations
give  
$$
\det \left[ \half \left( 1+ \left( \frac{i+j-1}{p} \right) \right) \right]_{1 \leq i,j \leq \frac{p-1}{2}}
=
\det \left[ \half \left( 1- \left( \frac{i+j-1}{p} \right) \right) \right]_{1 \leq i,j \leq \frac{p-1}{2}} ~=~-1
$$
for any prime $p>3$, $p \equiv 3 \!\! \pmod{4}$.
\cite{Chapman03} also includes additional conjectures related to such determinants.
In this paper, we prove the following 
evaluation of 
the Hankel determinant
$H_n(x)$: 
\begin{theorem}
\label{thm}
$H_n(x)$ identically vanishes unless $ n = p$ is a prime. For $p$ prime, 
$$
H_p (x) = (-1)^{\frac{p-1}{2}} p^{\frac{p-3}{2}}
\sum_{k=0}^{p-1} b_k x^k
$$
where
\begin{equation}
\label{bk}
 b_k =  \sum_{i=1}^{p-k} \left( \frac{ i}{p} \right) 
 ~.
\end{equation}
Furthermore, $ H_p(x) $ is divisible in $ \Z [x]$ by $ x^2$ for $ p  \equiv 3 \!\! \pmod{4}$ and 
by $ x^2 (x^2-1)$ for $ p  \equiv 1 \!\! \pmod{4}$. 
\end{theorem}

The properties of the Jacobi and Legendre symbols and Gauss sums that we make use of
in the proof of Theorem \ref{thm}
can readily be found in most books on number theory: we mention only \cite{BS66}, \cite{HR}, \cite{IR}.

\section{The proof of Theorem \ref{thm}}

We divide the proof of the theorem into a series of lemmas, and start with
recording the following trivial property of the polynomials $a_k(x)$:
\begin{lemma}
\label{stupidlemma}
$$
a_{k+1} ( x) = J( {k+1},{n})  + x a_k (x)  ~.
$$
\end{lemma}

\subsection{The composite case}

Now we show that 
$H_n(x) \equiv 0 $ iff $n$ is composite, and then determine the structure of 
$H_p(x)  $ for $ p $ prime.

\begin{lemma}
$H_n(x)$ identically vanishes for  $ n $  composite.
\end{lemma}
\begin{proof}
Let ${\bf r}_i = ( a_i, a_{i+1}, \ldots , a_{i+n-1} )$
denote the $i$-th row of the matrix in (\ref{Hn}). Let $ {\bf e}_i $ denote the
$n$-dimensional unit row vector with 1 in the $i$-th coordinate and 0 elsewhere, with
$ {\bf e}_i^t $ denoting its transpose.
The proof is in two cases depending on whether or not $n$ is a perfect square: \\

\noindent
{\bf Case I:} $n = m^2$ is a perfect square.\\

We claim that in this case the four rows 
${\bf r}_1 ,{\bf r}_2 ,{\bf r}_{m+1} , {\bf r}_{m+2}$ are linearly dependent. More precisely
$$
{\bf r}_2 -x {\bf r}_1 = {\bf r}_{m+2} - x {\bf r}_{m+1}
 ~.
$$
From Lemma \ref{stupidlemma}, 
\begin{equation}
\label{one}
{\bf r}_2 -x {\bf r}_1 = \sum_{i=1}^{m^2}
 J( {i+1},{m^2} )  {\bf e}_i
\end{equation}
and
\begin{equation}
\label{two}
{\bf r}_{m+2} -x {\bf r}_{m+1} = \sum_{i=1}^{m^2}
 J( {i+m+1},{m^2} )  {\bf e}_i
 ~.
\end{equation}
Note that
$$
 J( {a},{m^2} ) = 
 J( {a},{m} ) 
 J( {a},{m} ) = 
\left\{
\begin{array}{ll}
0 & \mbox{if} ~~~gcd(a,m) > 1, \\
1 & \mbox{if} ~~~gcd(a,m) = 1 ~.
\end{array}
\right.
$$
Since
$$
 gcd({i+1},{m} )  =
 gcd({i+m+1},{m}) 
 ~,
$$
the right hand sides of  (\ref{one}) and (\ref{two}) evaluate to the identical 0-1 vector. \\

\noindent
{\bf Case II:} $n = p^{2e+1} q$ with $p$ prime,  $ p \not  | ~q$. \\

Let $m = p^{2e+1}$. In this case we show that the following linear dependence among the rows holds:
$$
\sum_{i=0}^{p-1} 
({\bf r}_{iq+2} - x {\bf r}_{iq+1})  ~= {\bf 0}
 ~.
$$
By Lemma \ref{stupidlemma} the $j$-th entry of the vector on the left is
\begin{eqnarray*}
\sum_{i=0}^{p-1} 
 J( {iq+j+1}, {m q} )  
&=& J( {j+1},{ q} )  \sum_{i=0}^{p-1} J( {iq+j+1},{m } )  \\
&=& J( {j+1},{ q} )  \sum_{i=0}^{p-1} J( {iq+j+1},{p } )  \\
&=& J( {j+1},{ q} )  \sum_{i=1}^{p-1} \left( \frac{i}{p } \right) =0  ~.
\end{eqnarray*}
\end{proof}

\subsection{The prime case}

Let now $ n= p $ be prime.
Using Lemma \ref{stupidlemma} and 
replacing ${\bf r}_{i+1}$ by ${\bf r}_{i+1} - x {\bf r}_i$ for $i = 1, 2, \ldots, p$, we obtain
\dsp
\begin{equation}
\label{reduced}
H_p (x) = \det 
\left[
\begin{array}{ccccc}
a_1(x) & a_2(x) & \cdots &  a_{p-1} (x) & a_p(x) \\
\left( \frac{2}{p} \right) & \left( \frac{3}{p} \right) & \cdots & \left( \frac{p}{p} \right) & \left( \frac{p+1}{p} \right) \\
\left( \frac{3}{p} \right) & \left( \frac{4}{p} \right) & \cdots & \left( \frac{p+1}{p} \right) & \left( \frac{p+2}{p} \right) \\
\vdots & \vdots & \cdots & \vdots & \vdots \\
\left( \frac{p}{p} \right) & \left( \frac{p+1}{p} \right) & \cdots & \left( \frac{2p-2}{p} \right) & \left( \frac{2p-1}{p} \right) 
\end{array}
\right] ~.
\end{equation}
\sp

Since $a_k(x)$ is of degree $k -1$, $H_p (x)$ is a polynomial of degree 
$p-1$. 

Consider the $p \times p $ matrix
\dsp
\begin{equation}
\label{general}
{\bf A}_p = 
\left[
\begin{array}{cccc}
\left( \frac{1}{p} \right) & \left( \frac{2}{p} \right) & \cdots & \left( \frac{p}{p} \right)  \\
\left( \frac{2}{p} \right) & \left( \frac{3}{p} \right) & \cdots & \left( \frac{p+1}{p} \right)  \\
\vdots & \vdots & \cdots  & \vdots \\
\left( \frac{p}{p} \right) & \left( \frac{p+1}{p} \right) & \cdots & \left( \frac{2p-1}{p} \right)  
\end{array}
\right] = \left[ \left( \frac{i+j-1}{p} \right) \right]_{1\leq i,j \leq p}
\end{equation}
\sp

Let $c_{i,j}$ denote the cofactor of the entry $(i, j )$ of ${\bf A}_p$. Expanding the determinant in 
(\ref{reduced})
by the first row, we have 
$$
H_p (x) = \sum_{j=1}^p  c_{1,j} a_j(x)
$$
and 
the coefficient of the leading term is
the cofactor 
$ c_{1,p} = \det {\bf C}_p $ where
$ {\bf C}_p$ is the $ (p-1) \times (p-1)$ matrix
\dsp
\begin{equation}
\label{lterm}
{\bf C}_p = 
\left[
\begin{array}{cccc}
\left( \frac{2}{p} \right) & \left( \frac{3}{p} \right) & \cdots & \left( \frac{p}{p} \right)  \\
\left( \frac{3}{p} \right) & \left( \frac{4}{p} \right) & \cdots & \left( \frac{p+1}{p} \right)  \\
\vdots & \vdots & \cdots  & \vdots \\
\left( \frac{p}{p} \right) & \left( \frac{p+1}{p} \right) & \cdots & \left( \frac{2p-2}{p} \right)  
\end{array}
\right] = \left[ \left( \frac{i+j}{p} \right) \right]_{1\leq i,j \leq p-1}
\end{equation}
\sp

First we show that the $c_{1,j}$'s, and in fact all cofactors of ${\bf A}_p$ are identical. 
\begin{lemma}
\label{cofactor}
All cofactors of the matrix  ${\bf A}_p$ are identical.
\end{lemma}
\begin{proof}
We note that ${\bf A}_p$ is a symmetric matrix with the $i$-th row sum
$$
\sum_{j=1}^p \left( \frac{i+j-1}{p} \right) = 0
$$
for every $i$. Since the row sums vanish, the cofactor $c_{i,j}$ is independent of $j$. By symmetry, 
$c_{i,j}$ is also independent of $i$.
One way to prove Lemma \ref{cofactor} combinatorially is to use the standard weighted version of 
Kirchoff's matrix-tree theorem \cite{M}, \cite{S}. We include it here for completeness.
Consider the complete graph $K_p$ on vertices $\{1,2, \ldots, p \}$, and introduce the indeterminates 
$x_{i,j}$ for $ 1 \leq i , j \leq p$. Define
$$
D_i = \sum _{j=1}^p x_{i,j} ~ - x_{i,i}
$$
and define the weighted Laplacian matrix by setting ${\bf L}_p = [ L_{i,j}]_{1 \leq i,j \leq p}$,  with
$$
L_{i,j} = \left\{
\begin{array}{rl}
D_i & ~~ \mbox{if} ~~ i =j\\
-x_{i,j}  & ~~ \mbox{if} ~~ i \neq j
\end{array}
\right.
$$
Let $Sp \, (K_p)$ denote the set of spanning trees of $K_p$.
For any given index $i$, we can consider a $T \in Sp \, (K_p)$ as being {\em rooted} 
at vertex $i$. This simply gives 
an orientation to each edge $e = \{r, s \} $ of $T$ by orienting it from $r$ to $s$ iff
$s$ is closer to the root than $r$ in $T$. Define the weight of 
$e \in T $ by $ w_i(e) = x_{r,s}$ and the weight of $T$ itself by
$$
w_i(T) = \prod_{e \in T} w_i (e)
 ~.
$$
Then 
any cofactor $ c_{i,j}$ of an element in the $i$-th row of ${\bf L}_p$ is identical and evaluates to
\begin{equation}
\label{cij}
c_{i,j} = 
\sum_{T \in Sp \, (K_p)} w_i(T)  ~.
\end{equation}
This is the content of the weighted generalization of Kirchoff's matrix-tree theorem.
Suppose we specialize each $ x_{r,s}$, $r \neq s$ to a numerical value such that the resulting matrix
is symmetric (i.e $ x_{r,s}$ and $  x_{s,r}$ are assigned the same value).
Given a $ T \in Sp \, (K_p)$,  $w_i(T)$ then specializes to a fixed value independent of $i$ since 
the symmetry of the matrix implies that 
either edge orientation results in the same numerical weight for the edge.
Therefore the sum in (\ref{cij}) evaluates to the same quantity  independently of 
$ i,j$.
\end{proof}

Since
$ c_{1,j} = \det {\bf C}_p $ for all $j$,
we have proved
\begin{lemma}
$$
H_p (x) = \det {\bf C}_p 
 \sum_{j=1}^p a_j(x) ~.
$$
\end{lemma}

Note that
$$
 \sum_{j=1}^p a_j(x) = 
\sum_{k=0}^{p-1} b_k x^k
$$
where $b_k$ is as given in (\ref{bk}).
Next we evaluate $\det {\bf C}_p$.
\begin{lemma}
\label{lemma5}
\begin{equation}
\label{detAp}
\det {\bf C}_p = (-1)^{\frac{p-1}{2}} p^{\frac{p-3}{2}} ~.
\end{equation}
\end{lemma}
\begin{proof}
Let ${\bf E}_p $ denote the $p \times p $  exchange matrix which has 
1's along the
anti-diagonal and 0's elsewhere.
Clearly, 
$$
\det {\bf E}_p = (-1)^{\frac{p(p-1)}{2}}~.
$$
Let ${\bf B}_p = {\bf C}_p {\bf E}_p$. Then 
$$
{\bf B}_p = \left[ \left( \frac{ i-j}{p} \right) \right]_{1 \leq i,j \leq p-1}
$$
and
\begin{equation}
\label{CB}
\det {\bf C}_p = (-1)^{\frac{p-1}{2}} \det {\bf B}_p ~.
\end{equation}
Note that 
${\bf B}_p$ is symmetric for $ p  \equiv 1 \!\! \pmod{4}$ and
skew-symmetric for $ p  \equiv 3 \!\! \pmod{4}$.

We determine the spectrum of ${\bf B}_p$, and compute 
$ \det {\bf B}_p$ as the product of its eigenvalues. This results in the evaluation of 
$\det {\bf C}_p $ that we need through (\ref{CB}).

Let $ I = \sqrt{-1}$ and $ \zeta= e^{\frac{2 \pi I }{p}} $ denote a primitive $p$-th root of unity. 
For $ 1 \leq r \leq p-1$  
consider the Gauss sum
$$
g_r = \sum_{j=0}^{p-1} \left( \frac{j}{p} \right)  \, \zeta^{rj} ~.
$$
Then
\dsp
\begin{equation}
\label{eval}
g_r = \left\{
\begin{array}{rl}
 \left( \frac{r}{p} \right) \sqrt{p}  &~~\mbox{if}~~ p  \equiv 1 \!\! \pmod{4} \\
  I \left( \frac{r}{p}   \right) \sqrt{p}  &~~\mbox{if}~~ p  \equiv 3 \!\! \pmod{4} 
\end{array}
\right.
\end{equation}
\sp
A proof of Gauss's evaluation of $g_r$ can be found in \cite{IR}.
Changing the summation index, we can write
\begin{equation}
\label{eval2}
g_r = \sum_{j=0}^{p-1} \left( \frac{i-j}{p} \right)  \, \zeta^{r(i-j)} ~.
\end{equation}
We will give the details of the  proof for primes of the form
$p  \equiv 1 \!\! \pmod{4}$. 
The proof for primes $p  \equiv 3 \!\! \pmod{4}$ is similar.

For 
$p  \equiv 1 \!\! \pmod{4}$, $ \frac{p-1}{2}$ is even, and $ \det {\bf C}_p = \det {\bf B}_p $.
Using (\ref{eval}) and (\ref{eval2}), we have
$$
\sum_{j=0}^{p-1} \left( \frac{i-j}{p} \right)  \, \zeta^{-rj} = 
 \left( \frac{r}{p} \right) \sqrt{p}    \, \zeta^{-ri}
$$
or 
\begin{equation}
\label{main}
\left( \frac{i}{p} \right)  + 
\sum_{j=1}^{p-1} \left( \frac{i-j}{p} \right)  \, \zeta^{-rj} = 
 \left( \frac{r}{p} \right) \sqrt{p}    \, \zeta^{-ri} ~.
\end{equation}
Equating the imaginary parts in (\ref{main}),
$$
\sum_{j=1}^{p-1} \left( \frac{i-j}{p} \right) \sin \frac{2 \pi rj}{p} = 
 \left( \frac{r}{p} \right) \sqrt{p}    \,
 \sin \frac{2 \pi ri}{p}  ~.
$$
Therefore for every $r$ which is not zero modulo $p$, the vector
$$
u_r = \sum_{j=1}^{p-1} 
 \left( \sin \frac{2 \pi rj}{p} \right)  {\bf e}_j^t
$$
is an eigenvector of ${\bf B}_p$ corresponding to eigenvalue
$ \left( \frac{r}{p} \right) \sqrt{p} $. The vectors corresponding to $r$ and $p-r$
differ only in sign. Therefore if we let
$$
T_1 =  \{ u_r ~|~ 1 \leq r \leq \frac{p-1}{2} \}
$$
then exactly half of the $ u_r \in T_1$ are eigenvectors of ${\bf B}_p$ corresponding to 
eigenvalue $ \sqrt{p}$, and the other half are eigenvectors corresponding to 
the eigenvalue $ -\sqrt{p}$.
For $ 1 \leq r \leq s \leq \frac{p-1}{2}$, we have the trigonometric identity 
$$
  \sum_{j=1}^{p-1} 
 \sin \frac{2 \pi rj}{p} 
 \sin \frac{2 \pi sj}{p} 
=
\left\{
\begin{array}{ll}
 0 &~~\mbox{if}~  r < s \\
\frac{p}{2} 
  &~~\mbox{if}~  r = s
\end{array}
\right.
$$
where the $ r= s $ evaluation is a consequence of the 
general trigonometric identity 
\begin{equation}
\label{twin}
\sum_{j=1}^n \sin^2 j x = \frac{n}{2} - \frac{\cos(n+1) x \sin nx}{2 \sin x}
\end{equation}
(\cite{GR}, p. 30).
Therefore
the $ \frac{p-1}{2}$ eigenvectors in $T_1$ are orthogonal, and so linearly independent.

Next we obtain a set of $ \frac{p-1}{2} - 2 $ more eigenvectors of ${\bf B}_p$.
Equating the real parts in (\ref{main}), we obtain
\begin{equation}
\label{cos}
\left( \frac{i}{p} \right)  + 
\sum_{j=1}^{p-1} \left( \frac{i-j}{p} \right) \cos \frac{2 \pi rj}{p} = 
 \left( \frac{r}{p} \right) \sqrt{p} \,   
 \cos \frac{2 \pi ri}{p}  ~.
\end{equation}
Let 
$$
v_r = \sum_{j=1}^{p-1} 
 \left( \cos \frac{2 \pi rj}{p} \right)  {\bf e}_j^t ~.
$$
These are not themselves eigenvectors because of the extra 
term $ \left( \frac{i}{p} \right) $ in (\ref{cos}).
But the nonzero vectors of the form 
\begin{equation}
\label{nonzero}
v_r - v_s
\end{equation}
for $ 1 \leq r < s \leq p-1$ are eigenvectors of ${\bf B}_p$ as long as 
$ \left( \frac{r}{p} \right) = \left( \frac{s}{p} \right) $. 
We will single out
$$
\frac{p-1}{2} -2 
$$
of these eigenvectors, half corresponding to the eigenvalue $ \sqrt{p}$, and the other half to 
$ - \sqrt{p}$.
Let $g$ be a generator of the multiplicative group $ \Z_p^* $.
Then $ \left( \frac{g}{p} \right) =  -1 $.  Put $ h = g^4$.
Since 
$$
1 =  \left( \frac{1}{p} \right) = \left( \frac{h^k}{p} \right) ,
$$
taking $ r =1$, the $ \frac{p-1}{4}-1$ vectors
$$
v_1  - v_{h^k}
$$
for $ k = 1,  2, \ldots , \frac{p-1}{4}-1$ are eigenvectors of ${\bf B}_p$ corresponding to the 
eigenvalue $ \sqrt{p}$.
Similarly, 
$$
- 1 =  \left( \frac{g}{p} \right) = \left(  \frac{ g h^k}{p} \right) ,
$$
and
taking $ r =g$ in (\ref{nonzero}), the $ \frac{p-1}{4} -1$ vectors
$$
v_g  - v_{g h^k}
$$
for $ k = 1, 2, \ldots , \frac{p-1}{4}-1$ are eigenvectors of ${\bf B}_p$ corresponding to the 
eigenvalue $ - \sqrt{p}$.
These eigenvectors are of the form 
$$
v_1  - v_{h^k} =
\sum_{j=1}^{p-1} 
 \left( \cos \frac{2 \pi j}{p} 
 - \cos \frac{2 \pi h^k j}{p} \right)   {\bf e}_j^t
$$
in the first case, and 
$$
v_g  - v_{g h^k} = 
\sum_{j=1}^{p-1} 
 \left( \cos \frac{2 \pi g j}{p} - \cos \frac{2 \pi g h^k j}{p} \right)   {\bf e}_j^t
$$
in the second.
Let
$$
T_2 = \{ v_1  - v_{h^k} ~|~  k = 1, 2, \ldots , \frac{p-1}{4}-1 \} \cup
 \{ v_g  - v_{gh^k} ~|~  k = 1, 2, \ldots , \frac{p-1}{4}-1 \}  ~.
$$
Finally, consider  the two vectors
\begin{eqnarray*}
w_1 &=& \sum_{j=1}^{p-1} \half \left( 1 - \left( \frac{j}{p} \right) \right) {\bf e}_j^t \\
w_2 &=& \sum_{j=1}^{p-1} \half \left( 1 + \left( \frac{j}{p} \right) \right) {\bf e}_j^t  ~.
\end{eqnarray*}
Thus $w_1$ is a 0-1 vector with a 1 for every index for which the row sum of ${\bf B}_p$ is 1. Similarly,
$w_2$ is a 0-1 vector with a 1 for every index for which the row sum of ${\bf B}_p$ is $-1$.

The fact that $ w_1$ is an eigenvalue of ${\bf C}_p$ (and also 
of ${\bf B}_p$)
is a consequence of the identity
\dsp
$$
\sum_
{
\begin{array}{c}
j=1 \vspace*{-2mm}\\
\vspace*{-2mm}
( \frac{j}{p} ) = -1
\end{array}
}^{p-1}
 \left( \frac{i+j}{p} \right)  
=
\left\{
\begin{array}{ll}
 0 &~~\mbox{if}~  \left( \frac{i}{p} \right)  = 1 \\
 1 &~~\mbox{if}~  \left( \frac{i}{p} \right)  = -1 
\end{array}
\right.
$$
\sp
To prove this identity, write it in the form
$$
\sum_{j=1}^{p-1} 
 \left( \frac{i+j}{p} \right)  \half \left( 1 - \left( \frac{j}{p} \right) \right) = 
   \half \left( 1 - \left( \frac{i}{p} \right) \right)
$$
In this latter form
the identity can be proved by
expanding the left hand side and making use of
$$
\sum_{j=0}^{p-1} 
 \left( \frac{i+j}{p} \right)   \left( \frac{j}{p} \right)  =  -1
$$
which holds for $ p \not | ~i$
from the general orthogonality condition
\begin{equation}
\sum_{k=0}^{p-1} 
\left( \frac{i+k}{p} \right) 
\left( \frac{j+k}{p} \right) 
 = \left\{
\begin{array}{rl}
p-1 & \mbox{~~if~~} i = j\\
-1 & \mbox{~~if~~} i \neq j
\end{array}
\right.
\end{equation}
For ${\bf B}_p$, we obtain
$$
\sum_{j=1}^{p-1} 
 \left( \frac{i-j}{p} \right)  \half \left( 1 - \left( \frac{j}{p} \right) \right) = 
   \half \left( 1 - \left( \frac{i}{p} \right) \right) 
$$
so that $w_1$ is an eigenvector of ${\bf B}_p$ corresponding to eigenvalue $1$. Similarly, 
$w_2$ is an eigenvector of ${\bf B}_p$ 
corresponding to eigenvalue $ -1$.
Putting 
$$
T_3 = \{ w_1, w_2 \}
$$
we have $p-1$ eigenvectors in
$$
T_1 \cup T_2 \cup T_3
$$
with $\frac{p-2}{2}$ corresponding to 
eigenvalue $ \sqrt{p}$,
$\frac{p-2}{2}$ corresponding to 
eigenvalue $ -\sqrt{p}$, and one each for the eigenvalues $ \pm 1$.
To show that there is no linear dependence among these vectors, we proceed
to show that any two vectors $ u , v \in T_1 \cup T_2 \cup T_3$ are 
orthogonal.
We have already done this for 
$ u , v \in T_1 $. For $ u , v \in T_2$, we need to show 
\begin{eqnarray*}
\sum_{j=1}^{p-1}
 \left( \cos \frac{2 \pi j}{p} - \cos \frac{2 \pi h^r j}{p} \right)  
 \left( \cos \frac{2 \pi j}{p} - \cos \frac{2 \pi h^s j}{p} \right)  
& = & 0 \\
\sum_{j=1}^{p-1}
 \left( \cos \frac{2 \pi g j}{p} - \cos \frac{2 \pi g h^r j}{p} \right)  
 \left( \cos \frac{2 \pi g j}{p} - \cos \frac{2 \pi g h^s j}{p} \right)  
&= & 0 \\
\sum_{j=1}^{p-1}
 \left( \cos \frac{2 \pi j}{p} - \cos \frac{2 \pi h^r j}{p} \right)  
 \left( \cos \frac{2 \pi g j}{p} - \cos \frac{2 \pi g h^s j}{p} \right)  
&= & 0 
\end{eqnarray*}
for $ r \not \equiv s \!\! \pmod{p}$.

These identities follow from
$$
\sum_{j=1}^{n-1} \cos \frac{2 \pi r j }{n}
 \cos \frac{2 \pi s  j }{n}
 ~=~ \left\{
\begin{array}{ll}
 -1 &~~\mbox{if}~    s \neq r, n-r \\
 n-1 &~~\mbox{if}~    s = r= \frac{n}{2} \\
\frac{n-2}{2}
  &~~\mbox{if}~  s = 1 , n-r
\end{array}
\right.
$$
which holds for
$ 1 \leq r \leq s \leq n-1$,
and generalizes
the twin identity to (\ref{twin})
$$
\sum_{j=1}^n \cos^2 j x = \frac{n-1}{2} + \half \cos nx \sin (n+1)x \csc x
$$
(\cite{GR}, p. 31).

To prove that the vectors in $T_1$ are orthogonal to the vectors in $T_2$, we use 
the orthogonality relations
$$
\sum_{j=1}^{n-1} \cos \frac{2 \pi r j }{n}
 \sin \frac{2 \pi s  j }{n} ~=0
$$
valid for all integral $ r, s , n$.

Finally, below are the identities that are needed to prove that the 
vectors in $T_3$ are orthogonal to vectors in $ T_1$ and $T_2$.
If $p$ is a prime of the form $ 4k+1$, then

\begin{equation}
\label{T13}
\sum_{ \begin{array}{c} j=1  \vspace*{-1mm}\\ ( \frac{j}{p} ) = 1 \end{array} }^{p-1}
\sin \frac{2 \pi r j}{p} = 0 
\end{equation}
for any $r$, and 
\begin{equation}
\label{T23}
\sum_{ \begin{array}{c} j=1  \vspace*{-1mm}\\ ( \frac{j}{p} ) = 1 \end{array} }^{p-1}
\cos \frac{2 \pi r j}{p} =  \frac{-1+( \frac{r}{p} ) \sqrt{p} }{2}
\end{equation}
and
\begin{equation}
\label{T232}
\sum_{ \begin{array}{c} j=1 \vspace*{-1mm} \\ ( \frac{j}{p} ) = -1 \end{array} }^{p-1}
\cos \frac{2 \pi r j}{p} =  \frac{-1-( \frac{r}{p} ) \sqrt{p} }{2} ~.
\end{equation}

The first one of these can be written as 
$$
\sum_{j=1}^{p-1} \half \left( 1+\left( \frac{j}{p} \right) \right) \sin \frac{2 \pi r j}{p} = 0  ~.
$$
Clearly,
$$
\sum_{j=1}^{p-1} \sin \frac{2 \pi r j}{p} = 0 
$$
by looking at $ \sin$ as the imaginary part of $ \zeta$ and summing the geometric series in $ \zeta$. 
Therefore to prove  (\ref{T13}), it is enough to prove
$$
\sum_{j=1}^{p-1} \left( \frac{j}{p} \right)  \sin \frac{2 \pi r j}{p} = 0 
$$
which is an immediate consequence of the evaluation of the Gauss sum by equating 
the imaginary parts.

The identities  (\ref{T23}) and (\ref{T232}) are obtained by evaluating 
$$
\sum_{j=1}^{p-1}  \half \left( 1 \pm \left( \frac{j}{p} \right) \right)
\cos \frac{2 \pi r j}{p} 
$$
again
by making use of the evaluation of Gauss sums.
We sum the geometric series in $ \zeta$ and
equate the real parts.

Therefore the spectrum of ${\bf B}_p$ consists of $ \pm 1, \pm \sqrt{p}$ where 1 and $-1$ each have multiplicity one,
and $ \sqrt{p}$ and $ - \sqrt{p}$ each have multiplicity $ \frac{p-1}{2} - 1$. This gives
$$
\det {\bf B}_p = (-1)^{\frac{p-1}{2}} p^{\frac{p-3}{2}} ~.
$$
This completes the proof of Lemma \ref{lemma5}.
\end{proof}
I am grateful to the anonymous referee who suggested an alternate, and somewhat more economical proof
of Lemma \ref{lemma5}. I would like to sketch this approach here. 
Let 
$$
{\bf D}_p = \left[ \left( \frac{ i-j}{p} \right) \right]_{1 \leq i,j \leq p}
$$
We can view ${\bf B}_p$ as a submatrix of of ${\bf D}_p$ obtained by deleting the first row and column of ${\bf D}_p$.
${\bf D}_p$ is a circulant matrix, and therefore it has a basis of eigenvectors consisting of the 
$( 1 , \zeta^r, \zeta^{2r}, \ldots , \zeta^{(p-1)r})$. 
The eigenvalues are $g_r$: $(p-1)/2$ of them equal $g_1$, $(p-1)/2$ of 
them equal $  - g_1$ and also $g_0 =0$ must be included.
If we have an eigenvector of ${\bf D}_p$ with first entry zero, deleting that zero gives an eigenvector of 
${\bf B}_p$ with the same eigenvalue. Taking differences of the above basis elements gives $ (p-3)/2$ independent
eigenvectors of ${\bf B}_p$ with eigenvalue $g_1$, and $(p-3)/2$ with eigenvalue $ -g_1$. This accounts for 
all but two eigenvectors of ${\bf B}_p$, and these two are $ w_1$ and $w_2$.

\begin{remark}
For $p  \equiv 3 \!\! \pmod{4}$, the spectrum of 
${\bf B}_p$ consists of $ \pm I, \pm I \sqrt{p}$ where $I$ and $-I$ each have multiplicity one,
and $I \sqrt{p}$ and $ - I \sqrt{p}$ each have multiplicity $ \frac{p-1}{2} - 1$. 
In this case ${\bf B}_p$ is skew-symmetric, so the determinant is non-negative.
\end{remark}

\section{Special values}

We can obtain factors of $H_p(x)$ by finding zeros of 
\begin{equation}
\label{poly}
\sum_{k=0}^{p-1} b_k x^k
\end{equation}
where $b_k$ is as given in (\ref{bk}).

\begin{lemma}
For any $p$, $ x^2 ~|~ H_p(x)$.
\end{lemma}
\begin{proof}
It is easy to see that for any odd prime, $ b_0 = b_1 =0$.
Therefore $H_p(x)$ is divisible by $x^2$.
\end{proof}

Next we consider the case $p  \equiv 1 \!\! \pmod{4}$.
\begin{lemma}
\label{l2}
If $~p  \equiv 1 \!\! \pmod{4}$, then we also have
$ (x^2-1) ~|~ H_p(x)$.
\end{lemma}
\begin{proof}
The polynomial (\ref{poly}) evaluated at 
$ x=1$ and $ x= -1$ are 
\begin{eqnarray*}
\label{sums}
&& \sum_{k=1}^p \sum_{m=0}^k \left( \frac{k-m}{p} \right)  ~, \\ \nonumber
&& \sum_{k=1}^p \sum_{m=0}^k \left( \frac{k-m}{p} \right) (-1)^m 
\end{eqnarray*}
respectively. We will show that both of these evaluate to 0.
Rearranging the first sum,
$$
\sum_{k=1}^p \sum_{m=0}^k \left( \frac{k-m}{p} \right)  = 
\sum_{m=1}^{p-1}  (p-m+1) \left( \frac{m}{p} \right)   ~.
$$
Therefore it suffices to show that 
\begin{equation}
\label{suffices}
\sum_{m=1}^{p-1}  m \left( \frac{m}{p} \right) = 0 , 
\end{equation}
i.e. the sum of the quadratic residues minus the sum of the quadratic nonresidues mod $p$ vanishes.
Since
$$
 \left( \frac{-1}{p} \right) = (-1)^{\frac{p-1}{2}} = 1,  
$$
the map 
$ m \mapsto p-m$ permutes the quadratic 
residues among themselves, and the nonresidues among themselves. Since each
of these sets have an even number of elements for $p  \equiv 1 \!\! \pmod{4}$, this map has no fixed points.
Therefore both the residues and the nonresidues mod $p$ sum to
$$
 \frac{(p-1)}{4} p
$$
and (\ref{suffices}) follows.
The second sum in (\ref{sums}) can be rearranged as 
\begin{equation}
\label{alt}
\sum_{m=1}^{\frac{p-1}{2}} 
\left( \frac{2m-1}{p} \right)   
= - (-1)^{\frac{p^2-1}{8}  }
\sum_{m=1}^{\frac{p-1}{2}} 
\left( \frac{m}{p} \right)   
\end{equation}
and in this case 
the map 
$ m \mapsto p-m$ shows that there are equally many  residues mod $p$ in the range
$ \{ 1, 2, \ldots, \frac{p-1}{2} \}$ as in the range 
$\{ \frac{p-1}{2}+1, \ldots, p-1  \} $. A similar statement holds for nonresidues. 
Therefore the right hand side of (\ref{alt}) is zero and
$H_p (x)$ is divisible by $ x^2 -1  $ for $p  \equiv 1 \!\! \pmod{4}$.
\end{proof}

Note that the elementary arguments we gave for the proof of the evaluations in Lemma \ref{l2} can directly be 
obtained from the following result (see \cite{KW70}, also \cite{HD33}):
\begin{proposition}
Let $p$ be an odd prime and suppose $F$ is a complex-valued function defined on the integers, which is
periodic with period $p$. Then
$$
\sum_{j=0}^{p-1} F(j) +
\sum_{j=0}^{p-1} \left( \frac{j}{p} \right) F(j) =
\sum_{j=0}^{p-1} F(j^2)  ~.
$$
\end{proposition}

Finally, we remark that the coefficients of the quotient polynomials
$$
\frac{(-1)^{\frac{p-1}{2}} H_p (x) }{ p^{\frac{p-3}{2}} x^2 }
~~~ \mbox{and}~~~
\frac{(-1)^{\frac{p-1}{2}} H_p (x) }{ p^{\frac{p-3}{2}} x^2 (x^2-1)}
$$
over $\Z [x]$  can be written  in terms of the partial sums $ b_k$. For 
$p  \equiv 3 \!\! \pmod{4}$ these coefficients are simply $ b_{k+2}$. For
$p  \equiv 1 \!\! \pmod{4}$ the coefficients are partial sums of odd or even indexed $b_i$,  depending on 
the parity of $k$.

\begin{acknowledgment}
I would like to thank the anonymous referee who suggested an alternate proof of Lemma \ref{lemma5} and whose comments
greatly improved the presentation of this paper.
\end{acknowledgment}

\bibliographystyle{plain}

\begin{thebibliography}{1}

\bibitem{BS66}
Z. I. Borevich and I. R. Shafarevich,
\newblock {\em Number Theory}, Academic Press, 1966.

\bibitem{Chapman03}
R. Chapman,
\newblock Determinants of Legendre symbol matrices, 
\newblock {\em Acta Arith.}, 115 (2004), pp. 231--244.

\bibitem{HD33}
H. Davenport, 
\newblock On Certain Exponential Sums, {\em J. f\"{u}r Math.},  169 (1933), pp. 158-176.

\bibitem{GR}
I. S. Gradshteyn and I. M. Ryzhik,
{\em Table of Integrals, Series, and Products}, corrected and enlarged
edition, Academic Press Inc., Orlando, 1980.

\bibitem{HR}
G. H. Hardy and E. M. Wright,
\newblock {\em An introduction to the theory of numbers}, Fifth edition,
Oxford University Press, 1980. 

\bibitem{IR}
K. Ireland and M. Rosen, 
\newblock {\em A Classical Introduction to Modern Number Theory}, Second Edition, Springer-Verlag, 1990.

\bibitem{M}
J. W. Moon, 
\newblock {\em Counting labelled trees},  Canadian Mathematical Monographs 1, 
Canadian Mathematical Congress, Montreal, 1970.

\bibitem{S}
R. P. Stanley and S. Fomin, 
\newblock {\em Enumerative Combinatorics, Volume 2}, 
Cambridge University Press, 1999.

\bibitem{KW70}
K. S. Williams, 
\newblock Finite Transformation Formulae Involving the Legendre Symbol, 
{\em Pacific J. of Math.}, Vol 34, No. 2 (1970), pp. 559-568.

\end{thebibliography}

\end{document}